\documentclass[12pt, twoside]{amsart}
\usepackage{amssymb,amsmath,amsthm, amscd, enumerate, mathrsfs}
\usepackage{graphicx, hhline}
\usepackage[all]{xy}
\usepackage[usenames]{color}
\usepackage[usenames]{color}
\usepackage{hyperref}
\hypersetup{colorlinks=true}

\theoremstyle{plain}
\newtheorem{thm}{Theorem}[section] 
\newtheorem{cor}[thm]{Corollary}
\newtheorem{prop}[thm]{Proposition}

\newtheorem{lem}[thm]{Lemma}
\newtheorem*{mainthm}{Main Theorem}
\newtheorem*{mainconj}{Conjecture $\mathrm{B}_d$}
\newtheorem*{conjA}{Conjecture $\mathrm{A}_n$}
\newtheorem*{conjB}{Conjecture $\mathrm{B}_n$}
\theoremstyle{definition} 
\newtheorem{defn}[thm]{Definition}
\newtheorem{eg}[thm]{Example} 
\theoremstyle{remark}
\newtheorem{rem}[thm]{Remark}

\newtheorem*{cl}{Claim}
\newtheorem{cln}{Claim}
\newtheorem*{acknowledgement}{Acknowledgments}

\newcommand{\sO}{\mathcal{O}}
 
\newcommand{\Z}{\mathbb{Z}}
 
\newcommand{\Q}{\mathbb{Q}}

\newcommand{\m}{\mathfrak{m}}

\newcommand{\Hom}{\mathop{\mathrm{Hom}}\nolimits}
\newcommand{\Spec}{\mathop{\mathrm{Spec}}\nolimits}

\newcommand{\Supp}{\mathop{\mathrm{Supp}}\nolimits}

\newcommand{\bigzerol}{\smash{\lower1.0ex\hbox{\bg 0}}}

\renewcommand{\labelenumi}{\rm{(\theenumi)}}
\newfont{\bg}{cmr17 scaled\magstep5}

\title{On the $F$-purity of isolated log canonical singularities}
\author{Osamu Fujino}
\address{Department of Mathematics, Faculty of Science, Kyoto University, Kyoto 606-8502, Japan}
\email{fujino@math.kyoto-u.ac.jp}

\author{Shunsuke Takagi}
\address{Graduate School of Mathematical Sciences, University of Tokyo, 3-8-1 Komaba, Meguro-ku, Tokyo 153-8914, Japan}
\email{stakagi@ms.u-tokyo.ac.jp}

\keywords{log canonical singularities, $F$-pure singularities}
\subjclass[2010]{Primary 14B05; Secondary 13A35, 14E30}

\begin{document}

\begin{abstract}
A singularity in characteristic zero is said to be of \textit{dense $F$-pure type} if its modulo $p$ reduction is locally Frobenius split for infinitely many $p$. 
We prove that if $x \in X$ is an isolated log canonical singularity with $\mu(x \in X) \le 2$ (see Definition \ref{def1.4} for the definition of the invariant $\mu$), then it is of dense $F$-pure type. 
As a corollary, we prove the equivalence of log canonicity and being of dense $F$-pure type in the case of three-dimensional isolated $\Q$-Gorenstein normal singularities. 
\end{abstract}

\maketitle
\markboth{O.~FUJINO and S.~TAKAGI}{$F$-PURITY OF ISOLATED LOG CANONICAL SINGULARITIES}


\section*{Introduction}
A singularity in characteristic zero is said to be of dense $F$-pure type if its modulo $p$ reduction is locally Frobenius split for infinitely many $p$. 
The notion of strongly $F$-regular type is a variant of dense $F$-pure type and defined similarly using the Frobenius morphism after reduction to characteristic $p>0$ (see Definition \ref{dense Fpure def} for the precise definition).  
Recently it has turned out that they have a strong connection to singularities associated to the minimal model program. 
In particular, Hara \cite{Ha} proved that  a normal $\Q$-Gorenstein singularity in charscteristic zero is log terminal if and only if it is of strongly $F$-regular type. 
In this paper, as an analogous characterization for isolated log canonical singularities, we consider the following conjecture.
\begin{conjA}
Let $x \in X$ be an $n$-dimensional normal $\Q$-Gorenstein singularity defined over an algebraically closed field $k$ of characteristic zero such that $x$ is an isolated non-log-terminal point of $X$. 
Then $x \in X$ is log canonical if and only if it is of dense $F$-pure. 
\end{conjA}

Hara--Watanabe \cite{HW} proved that normal $\Q$-Gorenstein singularities of dense $F$-pure type are log canonical. 
Unfortunately, the converse implication is widely open and only a few special cases are known. 
For example, the two-dimensional case follows from the results of Mehta--Srinivas \cite{MeS} and Hara \cite{Ha1}, 
and the case of hypersurface singularities whose defining polynomials are very general was proved by Hern\'andez \cite{He}. 
This problem is now considered as one of the most important problems on $F$-singularities. 
Making use of recent progress on the minimal model program, we prove Conjecture $\mathrm{A}_3$. 

Let $x \in X$ be an $n$-dimensional isolated log canonical singularity defined over an algebraically closed field $k$ of characteristic zero. 
We suppose that $x \in X$ is not log terminal and $K_X$ is Cartier at $x$. 
Let $f: Y \to X$ be a resolution of singularities such that $f$ is an isomorphism outside $x$ and that $\mathrm{Supp}\; f^{-1}(x)$ is a simple normal crossing divisor on $X$. 
Then we can write 
$$K_Y=f^*K_X+F-E,$$
where $E$ and $F$ are effective divisors and have no common irreducible components. 
In \cite{Fu2}, the first author defined the invariant $\mu(x \in X)$ by
$$\mu=\mu(x \in X)=\min \{ \dim W \mid W \textup{ is a stratum of } E\}$$
and he showed that this invariant plays an important role in the study of $x \in X$. 
Using his method (which is based on the minimal model program), we can check that any minimal stratum $W$ of $E$ is a projective resolution of a $\mu$-dimensional projective variety $V$ with only rational singularities such that $K_V$ is linearly trivial. 
Also, running a minimal model program with scaling (see \cite{BCHM} for the minimal model program with scaling), we show that $H^{\mu}(V, \sO_V)$ can be viewed as the socle of the top local cohomology module $H^n_x(\sO_X)$ of $x \in X$. 

Here we introduce the following conjecture. 
\begin{mainconj}
Let $Z$ be a $d$-dimensional projective variety over an algebraically closed field of characteristic zero with only rational singularities such that $K_Z$ is linearly trivial. 
Then the action induced by the Frobenius morphism on the cohomology group $H^d(Z_p, \sO_{Z_p})$ of its modulo $p$ reduction $Z_p$ is bijective for infinitely many $p$. 
\end{mainconj}

Conjecture $\mathrm{B}_d$ is open in general, but it follows from a combination of the results of 
Ogus \cite{Og}, Bogomolov--Zarhin \cite{BZ} and Joshi--Rajan \cite{JR} that Conjecture $\mathrm{B}_d$ holds true if $d \le 2$. 

Now we suppose that Conjecture $\mathrm{B}_{\mu}$ is true.  
Applying Conjecture $\mathrm{B}_{\mu}$ to $V$, we see that the Frobenius action on the cohomology group $H^{\mu}(V_p, \sO_{V_p})$ of modulo $p$ reduction $V_p$ of $V$ is bijective for infinitely many $p$. 
On the other hand, by Matlis duality, the $F$-purity  of modulo $p$ reduction $x_p \in X_p$ of $x \in X$ is equivalent to the injectivity of the Frobenius action on $H^n_{x_p}(\sO_{X_p})$. 
This injectivity can be checked by the injectivity of the Frobenius action on its socle $H^{\mu}(V_p, \sO_{V_p})$. 
Thus, summing up the above, we conclude that $x \in X$ is of dense $F$-pure type.

A similar argument works in more general settings and our main result is stated as follows. 

\begin{mainthm}[=Theorem \ref{main result}]
Let $x \in X$ be a log canonical singularity defined over an algebraically closed 
field $k$ of characteristic zero
such that $x$ is an isolated non-log-terminal point of $X$. 
If Conjecture $\mathrm{B}_{\mu}$ holds true where $\mu=\mu(x \in X)$, then $x \in X$ is of dense $F$-pure type. 
In particular, if  $\mu(x \in X) \le 2$, then $x \in X$ is of dense $F$-pure type. 
\end{mainthm}

As a corollary of the above theorem, we show that Conjecture $\mathrm{A}_{n+1}$ is equivalent to Conjecture $\mathrm{B}_n$ (Corollary  \ref{equivalence of conjectures}). 
Since Conjecture $\mathrm{B}_2$ is known to be true, Conjecture $\mathrm{A}_3$ holds true. 
That is, log canonicity is equivalent to being of dense $F$-pure type in the case of three-dimensional isolated normal $\Q$-Gorenstein singularities (Corollary \ref{3-dim case}).   

\begin{acknowledgement}
The second author is grateful to Yoshinori Gongyo and Mircea Musta\c{t}\u{a} for helpful  conversations. 
The first and second authors were partially supported by Grant-in-Aid for Young Scientists (A) 20684001 
and (B) 23740024, respectively, from JSPS. 
\end{acknowledgement}

\section{Preliminaries on log canonical singularities}
In this section, we work over an algebraically closed field of characteristic zero. 
We start with the definition of singularities of pairs. 
Let $X$ be a normal variety and $D$ be an effective $\Q$-divisor on $X$ such that $K_X+D$ is $\Q$-Cartier. 
\begin{defn}
Let $\pi: \widetilde{X} \to X$ be a birational morphism from a normal variety $\widetilde{X}$.
Then we can write 
$$K_{\widetilde{X}}=\pi^*(K_X+D)+\sum_{E}a(E, X, D) E,$$ 
where $E$ runs through all the distinct prime divisors on $\widetilde{X}$ and $a(E, X, D)$ is a rational number. 
We say that the pair $(X, D)$ is \textit{canonical} (resp. \textit{plt}, \textit{log canonical}) 
if $a(E, X, D) \ge 0$ (resp. $a(E, X, D) >-1$, $a(E, X, D) \ge -1$) for every exceptional divisor $E$ over $X$.  
If $D=0$, we simply say that $X$ has only canonical (resp. log terminal, log canonical) singularities. 
We say that $(X, D)$ is \textit{dlt} if $(X, D)$ is log canonical and
there exists a log 
resolution $\pi: \widetilde{X} \to X$ such that $a(E, X, D)>-1$ for every 
$\pi$-exceptional divisor $E$ on $\widetilde{X}$. 
Here, a \textit{log resolution} $\pi: \widetilde{X} \to X$ of $(X, D)$ 
means that $\pi$ is a proper birational morphism, $\widetilde{X}$ is a smooth variety, 
$\mathrm{Exc}(\pi)$ is a divisor and $\mathrm{Exc}(\pi) \cup \mathrm{Supp}\; \pi^{-1}_*D$ 
is a simple normal crossing divisor. 
\end{defn}

\begin{defn}
A subvariety $W$ of $X$ is said to be a \textit{log canonical center} for the log canonical
pair $(X, D)$ if there exist a proper birational morphism $\pi: \widetilde{X} \to X$ 
from a normal variety $\widetilde{X}$ and a prime divisor $E$ on $\widetilde{X}$ with 
$a(E, X, D)= -1$ such that $\pi(E)=W$. 
Then $W$ is denoted by $c_{X}(E)$. 
\end{defn}

\begin{rem}\label{szabo}
Let $(X, D)$ be a dlt pair. 
There then exists a log resolution $f:Y \to X$ such that $f$ induces an isomorphism over the generic point of any log canonical center of $(X,D)$ and $a(E,X,D) >-1$ for every $f$-exceptional divisor $E$. 
This is an immediate consequence of  \cite[Divisorial Log Terminal Theorem]{Sz}. 
\end{rem}

From now on, let $X$ be a normal $\Q$-Gorenstein algebraic variety and $x \in X$ be a germ. 
The \textit{index} of $X$ at $x$ is the smallest positive integer $r$ such that $rK_X$ is Cartier at $x$. 

\begin{defn}\label{def1.4}
Let $x \in X$ be a log canonical singularity such that $x$ is a log canonical center. 
First we assume that the index of $X$ at $x$ is one. 
Take a projective birational morphism $f: Y \to X$ from a smooth variety $Y$ such that $\mathrm{Supp}\; f^{-1}(x)$ 
and $\mathrm{Exc}(f)$ are simple normal crossing divisors. Then we can write
$$K_Y=f^*K_X+F-E,$$
where $E$ and $F$ are effective divisors on $Y$ and have no common irreducible components. 
By assumption, $E$ is a reduced simple normal crossing divisor on $Y$.  
We define $\mu(x \in X)$ by
$$\mu(x \in X)=\min\{\dim W\mid W\textup{ is a stratum of $E$ and }f(W)=x\}.$$
Here we say a subvariety $W$ is a \textit{stratum} of $E=\sum_{i \in I}E_i$ if there exists a subset $\{i_1, \dots, i_k\} \subseteq I$ such that $W$ is an irreducible component of the intersection $E_{i_1} \cap \dots \cap E_{i_k}$. 
This definition is independent of the choice of the resolution $f$. 

In general, we take an index one cover $\rho: X' \to X$ with $x'=\rho^{-1}(x)$ to define $\mu(x \in X)$ by
$$\mu(x \in X)=\mu(x' \in X').$$
Since the index one cover is unique up to \'etale isomorphisms, the above definition of $\mu(x \in X)$ is well-defined. 
\end{defn}

We will give in Section \ref{sec4} a quick overview of the invariant $\mu$ and some related topics for the reader's convenience. 

\begin{rem}
(1) The first author showed in \cite[Theorem 5.5]{Fu2} that 
the invariant $\mu$ coincides with Ishii's Hodge theoretic invariant (see \cite{Is} and \cite[5.1]{Fu2} for the definition). 

(2) By the main result of \cite{Fu1}, the index of $x \in X$ is bounded if $\mu(x \in X) \le 2$. 
The reader is referred to \cite{Fu1} for the precise values of indices. 
\end{rem}

In order to prove the main result of this paper, we use the notion 
of \textit{dlt blow-ups}, which was first introduced by Christopher Hacon. 

\begin{lem}[\textup{cf.~\cite[Lemma 2.9]{Fu2} and \cite[Section 4]{fujino-ssmmp}}]\label{dlt blow-ups}
Let $X$ be a log canonical variety of index one such that $X$ is quasi-projective, 
$x$ is an isolated non-log-terminal point of $X$, and that $X$ is canonical outside $x$.  
Then there exists a projective birational morphism $g: Z \to X$ such that $K_Z+D=g^*K_X$ 
with $D$ a reduced divisor on $Z$, the pair $(Z, D)$ is a 
$\Q$-factorial dlt pair and $g$ is a small morphism  outside $x$.
\end{lem}

\begin{lem}\label{dlt blow-ups have only canonical singularities}
In Lemma \ref{dlt blow-ups}, $Z$ has only canonical singularities. 
\end{lem}
\begin{proof}
If $a(E, Z, D) > -1$, then $a(E, Z, D) \ge 0$ because $K_Z+D$ is Cartier. 
Since $K_Z$ is $\Q$-Cartier and $D$ is an effective divisor on $Z$, one has $a(E, Z, 0) \ge 0$.  
If $a(E, Z, D) =-1$, then we may assume that $Z$ is a smooth variety and $D$ is a reduced simple normal crossing divisor on $Z$ by shrinking $Z$ around the log canonical center $c_Z(E)$. 
In this case, $a(E, Z, 0) \ge 0$. Thus, $Z$ has only canonical singularities. 
\end{proof}


\section{Preliminaries on $F$-pure singularities}
In this section, we briefly review the definition of $F$-pure singularities and its properties which we will need later. 
\begin{defn}[\textup{\cite{HR}, \cite{HH2}}]\label{F-sing}
Let $x \in X$ be a point of an $F$-finite integral scheme $X$ of characteristic $p>0$. 
\renewcommand{\labelenumi}{(\roman{enumi})}
\begin{enumerate}
\item 
$x \in X$ is said to be \textit{$F$-pure} if the Frobenius map 
$$F: \sO_{X, x} \to F_*\sO_{X,x} \quad a \mapsto a^{p}$$ splits as an $\sO_{X, x}$-module homomorphism. 
\item
$x \in X$ is said to be \textit{strongly $F$-regular} if for every nonzero $c \in \sO_{X, x}$, there exists an integer $e \ge 1$ such that $$c F^e: \sO_{X, x} \to F^e_*\sO_{X,x} \quad a \mapsto c a^{p^e}$$ splits as an $\sO_{X, x}$-module homomorphism. 
\end{enumerate}
\end{defn}

\begin{rem}
Strong $F$-regularity implies $F$-purity. 
\end{rem}

The following criterion for $F$-purity is well-known to experts, but we include it here for the reader's convenience. 
\begin{lem}[\textup{cf.~\cite{HR}}]\label{Fpure_rem}
Let $x \in X$ be a closed point with index one of an $n$-dimensional $F$-finite integral scheme $X$.   
Then
$x \in X$ is $F$-pure if and only if $F(z) \ne 0$, where $F$ is the natural Frobenius action on $H^n_{x}(\sO_{X})$ and $z$ is a generator of the socle $(0:\m_x)_{H^n_{x}(\sO_{X})}$. 
\end{lem}

\begin{proof}
First note that $H^n_{x}(\sO_{X})$ is isomorphic to the injective hull of the residue field $\sO_{X, x}/\m_x$, because $\sO_{X, x}$ is quasi-Gorenstein. 
By definition, $x \in X$ is $F$-pure if and only if 
$$F^{\vee}: \Hom_{\sO_{X, x}}(F_*\sO_{X, x}, \sO_{X, x}) \to \Hom_{\sO_{X, x}}(\sO_{X, x}, \sO_{X, x})=\sO_{X, x}$$ is surjective. 
$F^{\vee}$ is the Matlis dual of the natural Frobenius action $F$ on $H^n_{x}(\sO_{X})$, so the surjectivity of $F^{\vee}$ is equivalent to the injectivity of $F$. 
Since $H^n_{x}(\sO_{X})$ is an essential extension of the socle $(0:\m_x)_{H^n_{x}(\sO_{X})}$, $F$ is injective if and only if $F|_{(0:\m_x)_{H^n_{x}(\sO_{X})}}$ is injective. 
Finally, the latter condition is equivalent to saying that $F(z) \ne 0$, because the socle $(0:\m_x)_{H^n_{x}(\sO_{X})}$ is a one-dimensional $\sO_{X, x}/\m_x$-vector space. 
\end{proof}

We define the notion of $F$-purity and strong $F$-regularity in characteristic zero, using reduction from characteristic zero to positive characteristic.  

\begin{defn}\label{dense Fpure def}
Let $x \in X$ be a point of a scheme of finite type over a field $k$ of characteristic zero. 
Choosing a suitable finitely generated $\Z$-subalgebra $A \subseteq k$, 
we can construct a (non-closed) point $x_A$ of a scheme $X_A$ of finite type over $A$ such that $(X_A, x_A) \times_{\Spec A} k \cong (X, x)$. 
By the generic freeness, we may assume that $X_A$ and $x_A$ are flat over $\Spec A$. 
We refer to $x_A \in X_A$ as a \textit{model} of $x \in X$ over $A$. 
Given a closed point $s \in \Spec A$, we denote by $x_{s} \in X_s$ the fiber of $x \in X$ over $s$.  
Then $X_s$ is a scheme defined over the residue field $\kappa(s)$ of $s$, which is a finite field. 
The reader is referred to \cite[Chapter 2]{HH} and \cite[Section 3.2]{MS} for more detail on reduction from characteristic zero to characteristic $p$. 

\renewcommand{\labelenumi}{(\roman{enumi})}
\begin{enumerate}
\item $x \in X$ is said to be of \textit{strongly $F$-regular type} if there exists a model of $x \in X$ over a finitely generated $\Z$-subalgebra $A$ of $k$ and a dense open subset $S \subseteq \Spec A$ such that $x_s \in X_s$ is strongly $F$-regular for all closed points $s \in S$. 
\item $x \in X$ is said to be of \textit{dense $F$-pure type} if there exists a model of $x \in X$ over a finitely generated $\Z$-subalgebra $A$ of $k$ and a dense subset of closed points $S \subseteq \Spec A$ such that $x_s \in X_s$ is $F$-pure for all $s \in S$. 
\end{enumerate}
\end{defn}

\begin{rem}
The definitions of strongly $F$-regular type and dense $F$-pure type are independent of the choice of a model. 
\end{rem}

\begin{thm}[\textup{\cite[Theorem 5.2]{Ha}}]
Let $x \in X$ be a normal $\Q$-Gorenstein singularity defined over a field of characteristic zero. 
Then $x \in X$ is log terminal if and only if it is of strongly $F$-regular type. 
\end{thm}

In this paper, we will discuss an analogous statement for log canonical singularities. 
Especially, we will consider the following conjecture. 

\begin{conjA}\label{lc vs Fpure}
Let $x \in X$ be an $n$-dimensional normal $\Q$-Gorenstein singularity defined over an algebraically closed field $k$ of characteristic zero such that $x$ is an isolated non-log-terminal point of $X$. 
Then $x \in X$ is log canonical if and only if it is of dense $F$-pure. 
\end{conjA}

\begin{rem}
Conjecture $\mathrm{A}_n$ is known to be true when $n=2$ (see \cite{Ha1}, \cite{MeS} and \cite{Wa}) or  when $x \in X$ is a hypersurface singularity whose defining polynomial is very general (see \cite{He}). 
The reader is referred to \cite[Remark 2.6]{Ta} for more detail. 
\end{rem}

\begin{defn}
Let $X$ be an $F$-finite scheme of characteristic $p>0$. 
If $X=\Spec R$ is affine,  we denote by $R[F]$ the ring
$$R[F] =\frac{R\{F\}}{\langle r^pF-Fr \mid r \in R \rangle}$$
which is obtained from $R$ by adjoining a non-commutative variable $F$ subject to the relation $r^pF=Fr$ for all $r \in R$. 
For a general scheme $X$, we denote by $\sO_X[F]$ the sheaf of rings obtained by gluing the respective rings $\sO_X(U_i)[F]$ over an affine open cover $X=\bigcup_i U_i$. 
\end{defn}

\begin{eg}\label{compatibility}
(1) 
Let $f: Y \to X$ be a morphism of schemes over an $F$-finite affine scheme $Z$. 
Then for all $i \ge 0$, $H^i(X, \sO_X)$ and $H^i(Y, \sO_Y)$ each has a natural $\sO_Z[F]$-module structure and $f$ induces an $\sO_Z[F]$-module homomorphism $f_*: H^i(X, \sO_X) \to H^i(Y, \sO_Y)$. 

(2)
Let $Y$ be a closed subscheme of  a scheme $X$ over an $F$-finite affine scheme $Z$. 
Then for all $i \ge 0$, we have the following natural exact sequence of $\sO_Z[F]$-modules 
$$\cdots \to H^i_Y(X, \sO_X) \to H^i(X, \sO_X) \to H^i(X \setminus Y, \sO_X) \to H^{i+1}_Y(X, \sO_X) \to \cdots.$$

(3)
Let $X$ be a scheme over an $F$-finite affine scheme $Z$ and  $Y_1, Y_2 \subseteq X$ be closed subschemes. 
Let $Y$ denote the scheme-theoretic union of $Y_1$ and $Y_2$. 
Then for all $i \ge 0$, the Mayer--Vietoris exact sequence
\begin{align*}
\cdots \to H^i(Y, \sO_{Y}) \to H^i(Y_1,  \sO_{Y_1}) \oplus H^i(Y_2, \sO_{Y_2}) &\to H^i(Y_1 \cap Y_2, \sO_{Y_1 \cap Y_2})\\ 
&\to H^{i+1}(Y, \sO_{Y}) \to \cdots
\end{align*}
becomes an exact sequence of $\sO_Z[F]$-modules. 
\end{eg}
\begin{proof}
The proof is immediate from the fact that every cohomology module in Example \ref{compatibility} can be computed from the  \v{C}ech complex. 
\end{proof}

The following proposition is a key to prove the main result of this paper. 

\begin{prop}\label{key lemma}
Let $x \in X$ be an $n$-dimensional normal singularity with index one defined over an algebraically closed field $k$ of characteristic zero. 
Let $g: Z \to X$ be a projective birational morphism and $D$ be a reduced $\Q$-Cartier divisor on $Z$ satisfying the following properties: 
\begin{enumerate}
\item $Z$ has only rational singularities,
\item  $K_Z+D \sim_g 0$,
\item  $g|_{Z \setminus D}: Z \setminus D \to X \setminus \{x\}$ is an isomorphism,
\item $\Supp D=\Supp {g}^{-1}(x)$, 
\end{enumerate}
Then $x \in X$ is of dense $F$-pure type if and only if given a model of $D$ 
over a finitely generated $\Z$-subalgebra $A$ of $k$, there exists a dense 
subset $S \subseteq \Spec A$ such that the action of Frobenius on $H^{n-1}(D_{s}, \sO_{D_s})$ 
is bijective for every closed point $s \in S$. 
\end{prop}

\begin{proof}
Without loss of generality, we may assume that $X$ is affine. 
Suppose given a model of $(x \in X, Z, D, g)$ over a finitely generated $\Z$-subalgebra $A$ of $k$. 

First we will show that enlarging $A$ if necessary, we can view $H^{n-1}(Z_s, \sO_{Z_s})$ as 
an $\sO_{X_s}[F]$-submodule of $H^n_{x_s}(\sO_{X_s})$ for all closed points $s \in \Spec A$. 
Since $f|_{Z \setminus D}: Z \setminus D \to X \setminus \{x\}$ is an isomorphism, we have natural isomorphisms
$$H^{n-1}(Z \setminus D, \sO_Z) \cong H^{n-1}(X \setminus \{x\}, \sO_X) \cong H^n_x(\sO_X).$$
On the other hand, we have the natural exact sequence
$$H^{n-1}_D(Z, \sO_Z) \to H^{n-1}(Z, \sO_Z) \to H^{n-1}(Z \setminus D, \sO_Z)$$
and $H^{n-1}_D(Z, \sO_Z)=0$ by the dual form of Grauert--Riemenschneider vanishing theorem (see, for example, \cite[Lemma 4.19 and Remark 4.20]{Fu3}). 
Hence we can view $H^{n-1}(Z, \sO_Z)$ as an $\sO_X$-submodule of $H^n_x(\sO_X)$.  
By Example \ref{compatibility} (1), (2), after possibly enlarging $A$, we may assume that $H^{n-1}(Z_s, \sO_{Z_s})$ is an $\sO_{X_s}[F]$-submodule of $H^n_{x_s}(\sO_{X_s})$ for all closed points $s \in \Spec A$. 

Next we will show that we may assume that 
$$H^{n-1}(Z_s, \sO_{Z_s}) \cong H^{n-1}(D_s, \sO_{D_s})$$ as an $\sO_{X_s}[F]$-module homomorphism for all closed points $s \in \Spec A$. 
The short exact sequence $0 \to \sO_Z(-D) \to \sO_Z \to \sO_D \to 0$ induces the exact sequence 
$$H^{n-1}(Z, \sO_Z(-D)) \to H^{n-1}(Z, \sO_Z) \to H^{n-1}(D, \sO_D) \to H^{n}(Z, \sO_Z(-D))=0$$
of $\sO_X$-modules.
It follows from the Grauert--Riemenschneider vanishing theorem that  $H^{n-1}(Z, \sO_Z(-D)) \cong H^{n-1}(Z, \sO_Z(K_Z))=0$, so we have an $\sO_X$-module isomorphism $H^{n-1}(Z, \sO_Z) \cong H^{n-1}(D, \sO_D)$.
By Example \ref{compatibility} (1), after possibly enlarging $A$, we 
may assume 
that $H^{n-1}(Z_s, \sO_{Z_s}) \cong H^{n-1}(D_s, \sO_{D_s})$ as an $\sO_{X_s}[F]$-module homomorphism for all  closed points $s \in \Spec A$. 

Finally, we will check that $H^{d-1}(D_s, \sO_{D_s})$ is the socle of the $\sO_{X_s, x_s}[F]$-module of  $H^d_{x_s}(\sO_{X_s})$. 
Since
$$\m_x \cdot H^{n-1}(Z, \sO_Z)=H^0(Z, \sO_Z(-D)) \cdot H^{n-1}(Z, \sO_Z) \subseteq H^{n-1}(Z, \sO_Z(-D))=0,$$
$H^{n-1}(D, \sO_{D}) \cong H^{n-1}(Z, \sO_Z)$ is contained in the socle of $H^n_x(\sO_X)$. 
Let $\omega_D$ be the dualizing sheaf of $D$. 
Then we obtain $\omega_D\simeq \mathcal O_Z(K_Z+D)\otimes \mathcal O_D$ since $K_Z+D$ is Cartier. 
Therefore, $\omega_D\simeq \mathcal O_D$ because $K_Z+D\sim _g0$ and $g(D)=x$. 
By Serre duality (which holds for top cohomology groups even if the variety 
is not Cohen--Macaulay), one has $\dim_k H^{n-1}(D, \sO_D)=1$,  
because $H^{0}(D, {\omega_D})=H^{0}(D, \sO_D) \cong k$. 
The socle of $H^n_{x}(\sO_{X})$ is the one-dimensional $k$-vector space, so it coincides with $H^{n-1}(D, \sO_{D})$. 

By the above argument, the bijectivity of the Frobenius action on $H^{n-1}(D_{s}, \sO_{D_s})$ means that the restriction of the Frobenius action on $H^n_{x_s}(\sO_{X_s})$ to its socle is injective. 
This condition is equivalent to saying that $X_s$ is $F$-pure by Lemma \ref{Fpure_rem}. 
Thus, $x_s \in X_s$ is of dense $F$-pure type if and only if there exists a dense subset of closed points $S \subseteq \Spec A$ such that the Frobenius action on $H^{n-1}(D_{s}, \sO_{D_s})$ is bijective for all $s \in S$. 
\end{proof}

\section{Main result}

In order to state our main result, we introduce the following conjecture.

\begin{conjB}
Let $V$ be an $n$-dimensional projective variety over an algebraically closed field $k$ of characteristic zero with only rational singularities  such that $K_V$ is linearly trivial. 
Given a model of $V$ over a finitely generated $\Z$-subalgebra $A$ of $k$, there exists a dense subset of closed points $S \subseteq \Spec A$ such that the natural Frobenius action on $H^n(V_{s}, \sO_{V_{s}})$ is bijective for every $s \in S$. 
\end{conjB}

\begin{rem}\label{remark on Conj B}
(1)
An affirmative answer to \cite[Conjecture 1.1]{MS} implies an affirmative answer to Conjecture $\mathrm{B}_n$. 
Indeed, take a resolution of singularities $\pi: \widetilde{V} \to V$. 
Since $V$ has only rational singularities, $\pi$ induces the 
isomorphism $H^n(V, \sO_V) \cong H^n(\widetilde{V}, \sO_{\widetilde{V}})$. 
Suppose given  a model of $\pi$ over a finitely generated $\Z$-subalgebra $A$ of $k$. 
If \cite[Conjecture 1.1]{MS} holds true, 
then there exists a dense subset of closed points $S \subseteq \Spec A$ 
such that the Frobenius action on $H^n(\widetilde{V}_{s}, \sO_{\widetilde{V}_{s}})$ is bijective for every $s \in S$. 
Since we may assume that 
$H^n(V_{s}, \sO_{V_{s}}) \cong H^n(\widetilde{V}_{s}, \sO_{\widetilde{V}_{s}})$ as 
$\kappa(s)[F]$-modules for all $s \in S$  by Example \ref{compatibility} (1), we obtain the assertion. 

(2) 
Let $W$ be an $n$-dimensional smooth projective variety defined over a perfect field of characteristic $p>0$. 
If $W$ is ordinary (in the sense of Bloch--Kato), then the natural Frobenius action on $H^n(W, \sO_W)$ is 
bijective (see \cite[Remark 5.1]{MS}). 
If $W$ is an abelian variety or a curve, then the converse implication also holds true (see \cite[Examples 5.4 and 5.5]{MS}).  
\end{rem}

\begin{lem}\label{n+1=>n}
Conjecture $\mathrm{B}_{n+1}$ implies Conjecture $\mathrm{B}_{n}$. 
\end{lem}

\begin{proof}
Let $V$ be an $n$-dimensional projective variety over an algebraically closed field $k$ of characteristic zero with only rational singularities such that $K_V$ is linearly trivial. 
Let $C$ be an elliptic curve over $k$, and denote $W=V \times C$. 
We suppose given a model of $(V, C, W)$ over a finitely generated $\Z$-subalgebra $A$ of $k$. 
Applying Conjecture $\mathrm{B}_{n+1}$ to $W$, we can take a dense subset of closed points $S \subseteq \Spec A$ such that the Frobenius action on 
$$H^{n+1}(W_s, \sO_{W_s})=H^n(V_s, \sO_{V_s}) \otimes H^1(C_s, \sO_{C_s})$$
is bijective for every $s \in S$. 
This implies that the Frobenius action on $H^n(V_s, \sO_{V_s})$ is bijective for every $s \in S$. 
\end{proof}

\begin{lem}\label{conj B2}
Conjecture $\mathrm{B}_n$ holds true if $n \le 2$. 
\end{lem}

\begin{proof}
By an argument similar to the proof of \cite[Proposition 5.3]{MS}, we may assume that $k=\overline{\Q}$ without loss of generality. 
By Lemma \ref{n+1=>n}, it suffices to consider the case when $n=2$. 

Let $\pi:\widetilde{X} \to X$ be a minimal resolution. 
$\widetilde{X}$ is an abelian surface or a K3 surface. 
Suppose given a model of $\pi$ over a finitely generated $\Z$-subalgebra $A$ of $k$. 
Then there exists a dense subset of closed points $S \subseteq \Spec A$ such that the Frobenius action on $H^2(\widetilde{X}_{s}, \sO_{\widetilde{X}_{s}})$ is bijective for every $s \in S$ (the abelian surface case follows from a result of Ogus \cite{Og} and the K3 surface case follows from a result of Bogomolov--Zarhin \cite{BZ} or that of Joshi and Rajan \cite{JR}).  
Since $X$ has only rational singularities, by Example \ref{compatibility} (1), we may assume that $H^2(X_{s}, \sO_{X_{s}}) \cong H^2(\widetilde{X}_{s}, \sO_{\widetilde{X}_{s}})$ as $\kappa(s)[F]$-modules for all $s \in S$.
Thus, we obtain the assertion.  
\end{proof}

Our main result is stated as follows. 
\begin{thm}\label{main result}
Let $x \in X$ be a log canonical singularity defined over an algebraically closed 
field $k$ of characteristic zero
such that $x$ is an isolated non-log-terminal point of $X$. 
If Conjecture $\mathrm{B}_{\mu}$ holds true where $\mu=\mu(x \in X)$, then $x \in X$ is of dense $F$-pure type. 
In particular, if  $\mu(x \in X) \le 2$, then $x \in X$ is of dense $F$-pure type. 
\end{thm}

We need the following proposition for the proof of Theorem \ref{main result}.  
\begin{prop}\label{MMP}
Let $x \in X$ be an $n$-dimensional log canonical singularity defined over an algebraically closed field $k$ of characteristic zero.
Suppose that the index of $X$ at $x$ is one and that $x$ is an isolated non-log-terminal point of $X$. 
Let $g:(Z, D) \to X$ be a dlt blow-up as in Lemma \ref{dlt blow-ups}. 
Then there exists a birational model $\widetilde{g}: (\widetilde{Z}, \widetilde{D}) \to X$ of $g$ which satisfies the following properties:
\begin{enumerate}
\item $\widetilde{Z}$ has only canonical singularities, 
\item $K_{\widetilde{Z}}+\widetilde{D}$ is linearly trivial over $X$,
\item $\widetilde{g}|_{\widetilde{Z} \setminus \widetilde{D}}:\widetilde{Z} \setminus \widetilde{D} \to X \setminus \{x\}$ is an isomorphism.  
\item $\Supp \widetilde{D}=\Supp {\widetilde{g}}^{-1}(x)$, 
\item given models of $D$ and $\widetilde{D}$ over a finitely generated $\Z$-subalgebra $A$ of $k$, 
enlarging $A$ if necessary, we may assume that 
$$H^{n-1}(D_{s}, \sO_{D_{s}}) \cong H^{n-1}({\widetilde{D}}_{s}, \sO_{{\widetilde{D}}_{s}})$$
as $\kappa(s)[F]$-modules for all closed points $s \in \Spec A$. 
\end{enumerate}
\end{prop}
\begin{proof}
We may assume that $X$ is affine and $K_X$ is Cartier. 
We run a $K_Z$-minimal model program over $X$ with scaling (see \cite{BCHM} for the minimal model program with scaling). 
Then we obtain a sequence of divisorial contractions and flips:
$$
\begin{xy}
(-10, 0) *{Z}="Z", (0,0) *{Z_0}="Z0", (18,0)*{Z_1}="Z1", (36,0)*{}="Z2", (46,0)*{}="Zk-2", (64,0)*{Z_{k-1}}="Zk-1", (82,0)*{Z_{k}}="Zk", (92,0)*{Z'}="Z'",
(-10,-10) *{D}="D", (0,-10) *{D_0}="D0", (18,-10)*{D_1}="D1", (36,-10)*{}="D2", (46,-10)*{}="Dk-2", (64,-10)*{D_{k-1}}="Dk-1", (82,-10)*{D_{k}}="Dk",(92,-10)*{D'}="D'",
\ar@{=} "Z"; "Z0"
\ar@{-->} "Z0"; "Z1"^{\phi_0}
\ar@{-->} "Z1"; "Z2"^{\phi_1}
\ar@{.} "Z2"; "Zk-2"
\ar@{-->} "Zk-2"; "Zk-1"^/-2mm/{\phi_{k-2}}
\ar@{-->} "Zk-1"; "Zk"^{\phi_{k-1}}
\ar@{=} "Zk"; "Z'"
\ar@{=} "D"; "D0"
\ar@{-->} "D0"; "D1"
\ar@{-->} "D1"; "D2"
\ar@{.} "D2"; "Dk-2"
\ar@{-->} "Dk-2"; "Dk-1"
\ar@{-->} "Dk-1"; "Dk"
\ar@{=} "Dk"; "D'"
{\ar@{^{(}->}  (0,-6.5); (0,-2)}
{\ar@{^{(}->}  (17,-6.5); (17,-2)}
{\ar@{^{(}->}  (59.5,-6.5); (59.5,-2)}
{\ar@{^{(}->}  (78,-6.5); (78,-2)}
\end{xy}
$$
such that $K_{Z'}$ is nef over $X$. 
Suppose given a model of the above sequence over a finitely generated $\Z$-subalgebra $A$ of $k$. 

\begin{cln}\label{flip}
Assume that $\phi_i: Z_i \dashrightarrow Z_{i+1}$ is a flip. 
Enlarging $A$ if necessary, we may assume that 
$$H^j(D_{i_s}, \sO_{D_{i_s}}) \cong H^j(D_{i+1,s}, \sO_{D_{i+1, s}})$$ 
as $\sO_{X_s}[F]$-modules for all $j$ and all closed points $s \in \Spec A$.  
\end{cln}
\begin{proof}[Proof of Claim \ref{flip}]
We consider the following flipping diagram:
$$
\xymatrix{
Z_i \ar@{-->}[rr]^{\phi_i} \ar[dr]_{\psi_i} & & Z_{i+1} \ar[dl]^{\psi_{i+1}}\\
& W_i & 
}
$$
Enlarging $A$ if necessary, we may assume that a model of the above diagram over $A$ is given. 
Note that $K_{Z_i}+D_i \sim_{\psi_i} 0$ and $K_{Z_{i+1}}+D_{i+1} \sim_{\psi_{i+1}} 0$.
Then we have the following exact sequences:
\begin{align*}
0 \to \sO_{Z_i}(K_{Z_i}) \to & \sO_{Z_i} \to \sO_{D_i} \to 0,\\
0 \to \sO_{Z_{i+1}}(K_{Z_{i+1}}) \to & \sO_{Z_{i+1}} \to \sO_{D_{i+1}} \to 0.
\end{align*}
We put $C_i=\psi_i(D_i)=\psi_{i+1}(D_i) \subseteq W_i$. 
Since $Z_i, Z_{i+1}$ and $W_i$ each has only rational singularities, by the Grauert--Riemenschneider vanishing theorem, one has 
$$\mathbf{R}{\psi_{i}}_*\sO_{D_i} \cong \sO_{C_i} \cong \mathbf{R}{\psi_{i}}_*\sO_{D_{i+1}}$$
in the derived category of coherent sheaves on $C_i$. 
Therefore, $\psi_i$ and $\psi_{i+1}$ induce the isomorphisms
$$H^j(D_i, \sO_{D_i}) \overset{{\psi_i}_*}{\cong} H^j(C_i, \sO_{C_i}) \overset{{\psi_{i+1}}_*}{\cong} H^j(D_{i+1}, \sO_{D_{i+1}})$$
for all $j$. 
By Example \ref{compatibility} (1), after possibly enlarging $A$, we may assume that 
$$H^j(D_{i_s}, \sO_{D_{i_s}}) \cong H^j(D_{i+1,s}, \sO_{D_{i+1, s}})$$ 
as $\sO_{X_s}[F]$-modules for all closed points $s \in \Spec A$. 
\end{proof}

\begin{cln}\label{divisorial contraction}
Assume that $\phi_i: Z_i \dashrightarrow Z_{i+1}$ is a divisorial contraction. 
Enlarging $A$ if necessary, we may assume that
$$H^j(D_{i_s}, \sO_{D_{i_s}}) \cong H^j(D_{i+1,s}, \sO_{D_{i+1, s}})$$ 
as $\sO_{X_s}[F]$-modules for all $j$ and all closed points $s \in \Spec A$.  
\end{cln}

\begin{proof}[Proof of Claim \ref{divisorial contraction}]
Let $E$ be the $\phi_i$-exceptional prime divisor on $Z_i$. 
First we will check that $\phi_i(D_i)=D_{i+1}$. 
It is obvious when $E$ is not an irreducible component of $D_i$, so we consider the case when $E$ is an irreducible component of $D_i$. 
Since $K_{Z_i}+D_i$ and $K_{Z_{i+1}}+D_{i+1}$ both are linearly trivial over $X$, we have 
$$K_{Z_i}+D_i=\phi_i^*(K_{Z_{i+1}}+D_{i+1}).$$
Hence $\phi_i(E)$ is a log canonical center of the pair $(Z_{i+1}, D_{i+1})$. 
Each $Z_i$ has only canonical singularities, because $Z$ has only canonical singularities by Lemma \ref{dlt blow-ups have only canonical singularities} and we run a $K_Z$-minimal model program. 
Thus, $\phi_i(E)$ has to be contained in $D_{i+1}$, which implies that $\phi_i(D_i)=D_{i+1}$. 

By an argument analogous to the proof of Claim $1$ (that is, by the Grauert--Riemenschneider vanishing theorem), we have $\mathbf{R}{\phi_i}_*\sO_{D_i} \cong \sO_{D_{i+1}}$ in the derived category of coherent sheaves on $D_{i+1}$. 
Therefore, $\phi_i$ induces the isomorphism
$$H^j(D_i, \sO_{D_i}) \overset{{\phi_i}_*}{\cong} H^j(D_{i+1}, \sO_{D_{i+1}})$$ 
for all $j$. 
By Example \ref{compatibility} (1), after possibly enlarging $A$, we may assume that 
$$H^j(D_{i_s}, \sO_{D_{i_s}}) \cong H^j(D_{i+1,s}, \sO_{D_{i+1, s}})$$ 
as $\sO_{X_s}[F]$-modules for all $s \in \Spec A$. 
\end{proof}

Let $g':(Z', D') \to X$ be the output of the minimal model program. 
By the base point free theorem, we obtain the following diagram
$$
\xymatrix{
(Z', D') \ar[rr]^{\pi} \ar[rd]_{g'} & & (\widetilde{Z}, \widetilde{D}) \ar[ld]^{\widetilde{g}}\\
& X &\\
}
$$
such that $\widetilde{Z}$ is the canonical model of $Z'$ over $X$ and that 
$$K_{Z'}+D'=\pi^*(K_{\widetilde{Z}}+\widetilde{D}).$$
Enlarging $A$ if necessary, we may assume that a model of the above diagram over $A$ is given. 
By an argument similar to the proof of Claim \ref{divisorial contraction},  we can check that $\pi(D')=\widetilde{D}$ and $\mathbf{R}\pi_*\sO_{D'} \cong \sO_{\widetilde{D}}$ in the derived category of coherent sheaves on $\widetilde{D}$. 
Thus, $\pi$ induces the isomorphism 
$$H^{n-1}(D', \sO_{D'}) \overset{\pi_*}{\cong} H^{n-1}(\widetilde{D}, \sO_{\widetilde{D}}).$$ 
By Example \ref{compatibility} (1), after possibly enlarging $A$, we may assume that 
$$H^{n-1}(D_{s}', \sO_{D_{s}'}) \cong H^{n-1}(\widetilde{D}_{s}, \sO_{\widetilde{D}_{s}})$$
as $\sO_{X_s}[F]$-modules for all closed points $s \in \Spec A$. 

Summing up the above arguments, we know that $\widetilde{g}:(\widetilde{Z}, \widetilde{D}) \to X$ has the following properties: 
\begin{enumerate}[(i)]
\item $\widetilde{Z}$ has only canonical singularities, 
\item $K_{\widetilde{Z}}+\widetilde{D} \sim_{\widetilde{g}} 0$,
\item $K_{\widetilde{Z}}$ is $\widetilde{g}$-ample,
\item $H^{n-1}(D_{s}, \sO_{D_{s}}) \cong H^{n-1}({\widetilde{D}}_{s}, \sO_{{\widetilde{D}}_{s}})$
for all closed points $s \in \Spec A$. 
\end{enumerate}
Since $-\widetilde{D}$ is $\widetilde{g}$-ample by (i) and (ii), one has $\Supp \widetilde{D}=\Supp \widetilde{g}^{-1}(x)$. 
Therefore, it remains to show that $\widetilde{g}$ is an isomorphism outside $x$.  
Note that ${X \setminus \{x \}}$ has only canonical singularities.
Then we can write 
$$K_{\widetilde{Z} \setminus \widetilde{D}}=g^*K_{X \setminus \{x \}}+F,$$
where $F$ is a $\widetilde{g}$-exceptional effective $\Q$-divisor on $\widetilde{Z} \setminus \widetilde{D}$.  
Since $K_{\widetilde{Z} \setminus \widetilde{D}}$ is $\widetilde{g}$-ample, one has $F=0$.   
Again, by the $\widetilde{g}$-ampleness of $K_{\widetilde{Z} \setminus \widetilde{D}}$, the birational morphism $\widetilde{g}: \widetilde{X} \setminus \widetilde{D} \to X \setminus \{x\}$ has to be finite, that is, an isomorphism.  
\end{proof}

Now we start the proof of Theorem \ref{main result}. 
\begin{proof}[Proof of Theorem \ref{main result}]
Since $F$-purity and log canonicity are preserved under index 
one covers (see \cite{Wa2} for $F$-purity and \cite[Proposition 5.20]{KM} for log canonicity), 
we may assume that the index of $X$ at $x$ is one. 
We can also assume that $X$ is affine and $K_X$ is Cartier.   

By Lemma \ref{dlt blow-ups}, there exists a birational projective morphism $g:Z \to X$ 
and a reduced divisor $D$ on $Z$ such that $K_Z+D=g^*K_X$, $(Z, D)$ is a $\Q$-factorial dlt pair 
and $g$ is a small morphism outside $x$. 
By Remark \ref{szabo}, there exists a projective birational morphism 
$h:Y \to Z$ from a smooth variety $Y$ with the following properties: 
\begin{enumerate}
\item $\mathrm{Exc}(h)$ and $\mathrm{Exc}(h) \cup \Supp h^{-1}_*D$ are simple normal crossing divisors on $Y$, 
\item $h$ is an isomorphism over the generic point of any log canonical center of the pair $(Z, D)$, 
\item $a(E, Z, D) >-1$ for every $h$-exceptional divisor $E$. 
\end{enumerate}
Then we can write 
$$K_Y=h^*(K_Z+D)+F-E,$$
where $E$ and $F$ are effective divisors on $Y$ which have no common irreducible components. 
By the construction of $h$, $E$ is a reduced simple normal crossing divisor on $Y$ and $E=h^{-1}_*D$. 
It follows from \cite[Corollary 4.15]{Fu3} or \cite[Corollary 2.5]{Fu2} that $\mathbf{R}h_*\sO_E \cong \sO_D$ in the derived category of coherent sheaves on $D$. 
Therefore, we have the isomorphism 
$$H^i(E, \sO_E) \overset{h_*}{\cong} H^i(D, \sO_D)$$ for every $i$. 
Suppose given models of $D$ and $E$ over a finitely generated $\Z$-subalgebra $A$ of $k$. 
By Example \ref{compatibility} (1), after possibly enlarging $A$, we may assume that 
$$H^i(E_s, \sO_{E_s}) \cong H^i(D_s, \sO_{D_s})$$ 
as $\sO_{X_s}[F]$-modules for all closed points $s \in \Spec A$. 

Let $W$ be a minimal stratum of a simple normal crossing variety $E$. 
By an argument similar to \cite[4.11]{Fu2} (see also Section \ref{sec4}), 
one has $\dim W=\mu$. 
Since $K_Z+D$ is linearly trivial over $X$ and $D$ is a $g$-exceptional divisor on $Z$, by the adjunction formula, one has $K_D \sim 0$. 
We also note that $D$ is sdlt (see \cite[Definition 1.1]{Fu0} for the definition of sdlt varieties). 
Applying \cite[Remark 5.3]{Fu2} to $h:E \to D$, we obtain the following claim.

\begin{cl}\label{MV sequence}
Suppose that models of $W$ and $E$ over $A$ are given.  
Then after possibly enlarging $A$, we may assume that 
$$H^{n-1}(E_s, \sO_{E_s}) \cong H^{\mu}(W_s, \sO_{W_s})$$
as $\sO_{X_s}[F]$-modules for all closed points $s \in \Spec A$. 
\end{cl}
\begin{proof}[Proof of Claim]
It follows from \cite[Theorem 5.2 and Remark 5.3]{Fu2} that  
$$H^{\mu}(W, \sO_{W}) \cong \dots \cong H^{n-1}(E, \sO_{E}),$$
where each isomorphism is the connecting homomorphism of a suitable Mayer--Vietoris exact sequence. 
Then by Example \ref{compatibility} (3), after possibly enlarging $A$, we may assume that 
$$H^{n-1}(E_s, \sO_{E_s}) \cong \cdots \cong H^{\mu}(W_s, \sO_{W_s})$$
as $\sO_{X_s}[F]$-modules for all closed points $s \in \Spec A$. 
\end{proof}

Let $V=h(W) \subseteq D$. Then $V$ is a minimal log canonical center of the pair $(Z, D)$. 
On the other hand, by adjunction formula for dlt pairs, we obtain $K_V=(K_Z+D)|_V \sim 0$. 
Thus, $V$ has only Gorenstein rational singularities. 
Since $h:W \to V$ is birational by the construction of $h$, one has the isomorphism 
$$H^\mu(W, \sO_W) \overset{h_*}{\cong} H^{\mu}(V, \sO_V).$$
Suppose models of $W$ and $V$ are given over $A$. 
By Example \ref{compatibility} (1),  after possibly enlarging $A$, we may assume that 
$$H^{\mu}(W_s, \sO_{W_s}) \cong H^{\mu}(V_s, \sO_{V_s})$$ 
as $\sO_{X_s}[F]$-modules for all closed points $s \in \Spec A$. 

Now we sum up the above arguments together with Proposition \ref{MMP} 
(we use the same notation as in Proposition \ref{MMP}). 
Suppose given models of $\widetilde{D}$ and $V$ over a finitely generated $\Z$-subalgebra $A$ of $k$. 
Then after possibly enlarging $A$, we may assume that 
$$H^{\mu}(V_s, \sO_{V_s}) \cong H^{n-1}(\widetilde{D}_s, \sO_{{\widetilde{D}}_s})$$ 
as $\sO_{X_s}[F]$-modules for all closed points $s \in \Spec A$. 
It follows from an application of Conjecture $\mathrm{B}_{\mu}$ to $V$ that there exists a dense subset of closed points $S \subseteq \Spec A$ such that the natural Frobenius action on $H^{\mu}(V_s, \sO_{V_s})$ is bijective for all $s \in S$. 
Then the Frobenius action on $H^{n-1}(\widetilde{D}_s, \sO_{{\widetilde{D}}_s})$ is also bijective for all closed points $s \in S$, which implies by Proposition \ref{key lemma} that $x \in X$ is of dense $F$-pure type. 
\end{proof}

\begin{rem} 
Let $f:Y\to X$ be any resolution as in Definition \ref{def1.4}. 
By the uniqueness of the relative canonical model, we have 
$$
\widetilde Z\cong \mathrm{Proj} \bigoplus _{m\geq 0} f_*\mathcal O_Y(mK_Y)
$$ 
over $X$. 
Unfortunately, by this construction, it is not clear how to relate the cohomology group  
$H^{n-1}(D_s, \mathcal O_{D_s})$ to $H^{n-1}(\widetilde D_s, \mathcal O_{{\widetilde D}_s})$. 
Moreover, the relationship between $\widetilde D$ and a minimal stratum of $E$ in Definition 
\ref{def1.4} is also not clear. 
Therefore, we take a dlt blow-up and run a minimal model program with scaling to construct $\widetilde Z$. 
\end{rem}

\begin{cor}\label{equivalence of conjectures}
Conjecture $\mathrm{A}_{n+1}$ is equivalent to Conjecture $\mathrm{B}_n$. 
\end{cor}
\begin{proof}
First we will show that Conjecture $\mathrm{B}_n$ implies Conjecture $\mathrm{A}_{n+1}$. 
Let $x \in X$ be an $(n+1)$-dimensional normal $\Q$-Gorenstein singularity defined over an algebraically closed field $k$ of characteristic zero such that $x$ is an isolated non-log-terminal point of $X$. 
If $x \in X$ is of dense $F$-pure type, then by \cite[Theorem 3.9]{HW}, it is log canonical. 
Conversely, suppose that  $x \in X$ is a log canonical singularity.  
Since $\mu:=\mu(x \in X) \le \dim X-1=n$, by Lemma \ref{n+1=>n}, Conjecture $\mathrm{B}_{\mu}$ holds true. 
It then follows from Theorem \ref{main result} that $x \in X$ is of dense $F$-pure type. 

Next we will prove that Conjecture $\mathrm{A}_{n+1}$ implies Conjecture $\mathrm{B}_{n}$. 
Let $V$ be an $n$-dimensional projective variety over an algebraically closed field $k$ of characteristic zero with only rational singularities such that $K_V \sim 0$. 
Take any ample Cartier divisor $D$ on $V$ and consider its section ring $R=R(V, D)=\bigoplus_{m \ge 0}H^0(V, \sO_V(mD))$. 
By \cite[Proposition 5.4]{SS}, the affine cone $\Spec R$ of $V$ has only quasi-Gorenstein log canonical singularities and its vertex is an isolated non-log-terminal point of $\Spec R$. 
It then follows from Conjecture $\mathrm{A}_{n+1}$ that given a model of $(V, D, R)$ over a finitely generated $\Z$-subalgebra $A$ of $k$, there exists a dense subset of closed points $S \subseteq \Spec A$ such that $\Spec R_s$ is $F$-pure for all $s \in S$. 
Note that after replacing $S$ by a smaller dense subset if necessary, we may assume that $R_s=R(V_s, D_s)$ for all $s \in S$. 
Since $\Spec R_s$ is $F$-pure, the natural Frobenius action on the local cohomology module $H^{n+1}_{\m_{R_s}}(R_s)$ is injective, where $\m_{R_s}=\bigoplus_{m \ge 1}H^0(V_s, \sO_{V_s}(mD_s))$ is the unique homogeneous maximal ideal of $R_s$.  
Then the Frobenius action on $H^n(V_s, \sO_{V_s})$ is also injective, because $H^n(V_s, \sO_{V_s})$ is the degree zero part of $H^{n+1}_{\m_{R_s}}(R_s)$. 
\end{proof}

Since Conjecture $\mathrm{B}_2$ is known to be true (see Lemma \ref{conj B2}), Conjecture $\mathrm{A}_3$ holds true.  
\begin{cor}\label{3-dim case}
Let $x \in X$ be a three-dimensional normal $\Q$-Gorenstein singularity defined over an algebraically closed field of characteristic zero such that $x$ is an isolated non-log-terminal point of $X$.  
Then $x \in X$ is log canonical if and only if it is of dense $F$-pure type.
\end{cor}

\section{Appendix:~A quick review of \cite{Fu1} and \cite{Fu2}}\label{sec4}

In this appendix, we quickly review the invariant $\mu$ and related topics in \cite{Fu1} and 
\cite{Fu2} for the reader's convenience. After \cite{Fu1} was written, 
the minimal model program has developed drastically (cf.~\cite{BCHM}). 
In \cite{Fu2}, we only treat isolated log canonical singularities. 
Here, we survey the basic properties of $\mu$ and some related results in the framework of \cite{Fu2}. 
For the details, see \cite{Fu1} and \cite{Fu2}. 
 
Let $X$ be a quasi-projective log canonical variety defined over an algebraically closed 
field $k$ of characteristic zero with index one. 
Assume that $x\in X$ is a log canonical center. 
Let $f:Y\to X$ be a projective birational morphism from a smooth variety $Y$ such that 
$$
K_Y=f^*K_X+F-E
$$
where $E$ and $F$ are effective Cartier divisors on $Y$ and have no common irreducible components. 
We further assume that 
$f^{-1}(x)$ and $\Supp (E+F)$ are simple normal crossing divisors on $Y$. 
Let $E=\sum _{i\in I}E_i$ be the irreducible decomposition. 
Note that $E$ is a reduced simple normal crossing divisor on $Y$. 
We put $$
J=\{i\in I\, |\, f(E_i)=x\}\subset I
$$ 
and $$G=\sum _{i\in J}E_i. 
$$
Then, by \cite[Proposition 8.2]{fujino-fund}, we obtain 
$$f_*\mathcal O_G\cong \kappa (x). $$
In particular, $G$ is connected. 
We apply a $(K_Y+E)$-minimal model program with scaling 
over $X$ (cf.~\cite{BCHM} and \cite[Section 4]{fujino-ssmmp}). 
Then we obtain a projective birational morphism 
$$
f':Y'\to X 
$$ 
such that $(Y', E')$ is a $\mathbb Q$-factorial dlt pair and that $K_{Y'}+E'=f'^*K_X$ 
where $E'$ is the pushforward of $E$ on $Y'$.  It is a dlt blow-up of $X$ 
(cf.~Lemma \ref{dlt blow-ups}). 
Note that each step of the minimal model program is an isomorphism 
at the generic point of any log canonical center of $(Y, E)$ because 
$$
K_Y+E=f^*K_X+F. 
$$ 
Therefore, we obtain 
$$
\mu(x\in X)=\min\{\dim W | \, W \ {\text{is a log canonical center of}} \ (Y', E')\  {\text{with}}\ 
f'(W)=x\}. 
$$
By the proof of \cite[Theorem 10.5 (iv)]{fujino-fund}, 
we have 
$$
f'_*\mathcal O_{G'}\cong \kappa (x)
$$ 
where $G'$ is the pushforward of $G$ on $Y'$. 
In particular, $G'$ is connected. 
By applying \cite[Proposition 3.3]{Fu2} to each irreducible component of $G'$, we can check that 
if $W$ is a minimal log canonical center of $(Y', E')$ with $f'(W)=x$ then $\dim W=\mu (x\in X)$. 
By this observation, every minimal stratum of $E$ which is mapped to $x$ by 
$f$ is $\mu(x\in X)$-dimensional and $\mu(x\in X)$ is independent of 
the choice of the resolution $f$ (cf.~\cite[4.11]{Fu2}), that is, $\mu(x\in X)$ is well-defined. 

Let $W_1$ and $W_2$ be any minimal log canonical centers of $(Y', E')$ such that 
$f'(W_1)=f'(W_2)=x$. Then 
we can check that $W_1$ is birationally equivalent to $W_2$ (cf.~\cite[Proposition 3.3]{Fu2}). 
Therefore, all the minimal stratum of $E$ mapped to $x$ by $f$ are birational each other. 
More precisely, we can take a common resolution 
\begin{equation*}
\xymatrix{ & W\ar[dl]_{\alpha_1} \ar[dr]^{\alpha_2}\\
W_1 \ar@{<-->}[rr]  & & W_2}
\end{equation*}
such that 
$\alpha_1^*K_{W_1}=\alpha_2^*K_{W_2}$ 
(cf.~\cite[Proposition 3.3]{Fu2}). 

By the adjunction formula for dlt pairs (cf.~\cite[Proposition 3.9.2]{fujino-what}), 
we can check that 
$$
K_W=(K_{Y'}+E')|_W\sim 0
$$ 
and that $W$ has only canonical Gorenstein singularities if 
$W$ is a minimal log canonical center of $(Y', E')$ with $f'(W)=x$. 

Let $V$ be any minimal stratum of $E$. 
Then we can prove that 
$$
H^\mu(V, \mathcal O_V)\overset {\delta}\cong H^{n-1}(E, \mathcal O_E)  
$$ 
when $f(E)=x$, equivalently, $x\in X$ is an isolated non-log-terminal point, where 
$\mu=\mu(x\in X)$ and $n=\dim X$. 
The isomorphism $\delta$ is a composition of connecting homomorphisms of suitable 
Mayer--Vietoris exact sequences.  
For the details, see \cite[Section 5]{Fu2}. 
Although we assume that the base field is $\mathbb C$ and 
use the theory of mixed Hodge structures in \cite{Fu2}, the above isomorphism holds over an 
arbitrary algebraically closed field $k$ of characteristic zero by the Lefschetz principle.

\end{document}